\documentclass{ws-m3as}
\usepackage{undertilde}
\usepackage{array}
\newcommand{\PreserveBackslash}[1]{\let\temp=\\#1\let\\=\temp}
\newcolumntype{C}[1]{>{\PreserveBackslash\centering}p{#1}}
\newcolumntype{R}[1]{>{\PreserveBackslash\raggedleft}p{#1}}
\newcolumntype{L}[1]{>{\PreserveBackslash\raggedright}p{#1}}

\begin{document}
\makeatletter
\def\@captype{figure}
\makeatother

\markboth{Yanyan Yu, Weihua Deng, Yujiang Wu}{Positivity and
boundedness preserving schemes for the fractional reaction-diffusion equation}

%
\catchline{}{}{}{}{}
%

\title{Positivity and boundedness preserving
semi-implicit schemes for the fractional reaction-diffusion equation}

\author{Yanyan Yu, Weihua Deng\footnote{Corresponding author. E-mail: dengwh@lzu.edu.cn},\,Yujiang Wu}
\address{School of Mathematics and Statistics, Lanzhou University, Lanzhou 730000,\\
 People's Republic of China}

\maketitle


\begin{abstract}

In this paper, we design a
semi-implicit scheme for the scalar time fractional
reaction-diffusion equation. We theoretically prove that the
numerical scheme is stable without the restriction on the ratio of
the time and space stepsizes, and numerically show that the convergent orders are $1$ 
in time and $2$ in space. As a
concrete model, the subdiffusive predator-prey system is
discussed in detail. First, we prove that the analytical solution of the
system is positive and bounded. Then we use the provided numerical
scheme to solve the subdiffusive predator-prey system, and
theoretically prove and numerically verify that the numerical scheme
preserves the positivity and boundedness.

%
%
%
\end{abstract}



\section{Introduction}

Mathematically, the reaction-diffusion systems take the form of
semi-linear parabolic partial differential equations. Usually, in
real world applications, the reaction term describes the birth-death
or reaction occurring inside the habitat or reactor. The diffusion
term models the movement of many individuals in an environment or
media. The individuals can be very small particles in physics,
bacteria, molecules, or cells, or very large objects such as
animals, plants. The diffusion is often described by a power law,
$\langle x^2(t) \rangle - \langle x(t) \rangle^2 \sim Dt^\alpha$,
where $D$ is the diffusion coefficient and $t$ is the elapsed
time.\cite{Metzler:00} In a normal diffusion, $\alpha=1$. If
$\alpha>1$, the particle undergoes superdiffusion, and it results from
 active cellular transport processes. If $\alpha<1$, the
phenomenon is called subdiffusion, it can be  protein diffusion
within cells, or diffusion through porous media. This paper concerns
the subdiffusive reaction-diffusion system, which corresponds to the
classical reaction-diffusion equation with its first order time
derivative replaced by the $\alpha-$th order fractional time
derivative.

As an important concrete example, we detailedly discuss the
subdiffusive predator-prey model. All living things within an
ecosystem are interdependent. A change in the size of one population
or the environment they live affects all other organisms within the
ecosystem. This is shown particularly clearly by the relationship
between predator and prey populations. Cavani and
Farkas\cite{Cavani1994} introduce diffusion to the
Michaelis-Menten-Holling predator-prey model. Then the more general
models are considered\cite{Bartumeus2001,Pang2003,Wang2003}.  Here
we further discuss the Michaelis-Menten-Holling predator-prey model
with the subdiffusive mechanism\cite{aly2011}:
\begin{equation}\label{Eq1.1}
\begin{array}{llll}
\displaystyle\frac{\partial^\alpha N}{\partial t^\alpha} &= & \displaystyle\frac{\partial^2 N}{\partial x^2}+N\left(1-N-\frac{a P}{P+N}\right), & x \in (l,r),\,\,t>0,\\
\\
\displaystyle\frac{\partial^\alpha P}{\partial t^\alpha} &= &
\displaystyle\frac{\partial^2 P}{\partial x^2}+\sigma P
\left(-\frac{\gamma+\delta \beta P}{1+\beta P}+\frac{N}{P+N}\right),
& x \in (l,r),\,\,t>0,
\end{array}
\end{equation}
with the positive initial conditions and the homogeneous Dirichlet
boundary conditions
\begin{equation}\label{Eq1.2}
N(l,t)=N(r,t)=P(l,t)=P(r,t)=0,
\end{equation}
or
the homogeneous Neumann  boundary conditions
$$
(\partial {N(x,t)}/\partial x) |_{x=l \rm{~and~} r, \rm{~respectively}}=(\partial {P(x,t)}/\partial x) |_{x=l \rm{~and~} r, \rm{~respectively}}=0, \eqno{(1.2^\prime)}
$$
where $a$, $\sigma$, and $\beta$  are positive real numbers, and $0<\gamma \le \delta$.
We prove that the analytical solution of (\ref{Eq1.1}) and
((\ref{Eq1.2}) or $(1.2^\prime)$) is positive and bounded.

For the analytical solution of the subdiffusion equation, the reader
can refer to Refs.
\refcite{Agrawal2002,Gorenflo2002,Mainardi1996,Schneider1989,Wyss1986},
and the references therein. There are also some works for the
numerical solutions of subdiffusion equations, e.g.,~Refs.
\refcite{Chen2007,Cui2009,Deng2007,Gao2011,Lin2007,Yuste2005}. In
particular, Zhang and Sun develop the semi-implicit schemes
for the subdiffusive equations\cite{Zhang2011}. For the past few
decades, the semi-implicit schemes are widely used in various
complicated time dependent non-linear equations. Usually the semi-implicit schemes use two time levels; in time level 1, the nonlinear terms are explicitly computed, and then to implicitly solve the high order linear terms. The expected advantage of the semi-implicit scheme is that as the nonlinear terms are computed efficiently but not losing good numerical stability.
Here, we
construct the semi-implicit scheme to numerically solve the
subdiffusive reaction-diffusion equation.
The stability of the numerical scheme is strictly proved, and
it has no restriction on the ratio of the sizes of space steps and
time ones. The convergent orders $1$ in time and $2$ in space
are theoretically obtained and numerically verified. Moreover, we
use the provided scheme to numerically solve the subdiffusive
predator-prey model. We show both theoretically and numerically that
it preserves the positivity and boundedness of the solutions of the
subdiffusive predator-prey model.

The outline of the paper is as follows. In Section 2, we propose the
time fractional semi-implicit  scheme for the subdiffusive
reaction-diffusion equation. We discuss the stability and
convergence of the proposed scheme in Section 3, and prove that the
temporal approximation order is $1$ 
and the order in space is
$2$. In Section 4, we first prove the positivity and boundedness of
the solution of the subdiffusive predator-prey model, then certify
that the numerical scheme preserves its positiveness and
boundedness. In Section 5, we perform the numerical experiments to
confirm the convergent orders and positivity and boundedness
preserving. We conclude the paper with some remarks in the last
section.

\section{Scheme for the subdiffusive reaction-diffusion equation}

We first consider the following scalar subdiffusive
reaction-diffusion equation:
\begin{equation}\label{2.1}
\frac{\partial^\alpha u(x,t)}{\partial t^\alpha} =\frac{\partial^2
u(x,t)}{\partial x^2}+f(u(x,t)),
\end{equation}
with $ x \in \Omega=(0,1) $, $ 0< t \leq T $, $0<\alpha<1$, the
initial condition
\begin{equation}\label{2.2}
u(x,0)=g(x),\qquad x\in\Omega,
\end{equation}
and the boundary conditions
\begin{equation}\label{2.3}
u(0,t)=u(1,t)=0,\quad 0\leq t\leq T,
\end{equation}
or
$$
({\partial u(x,t)}/{\partial x})|_{x=0}=({\partial u(x,t)}/{\partial x})|_{x=1}=0,\quad 0\leq t\leq T,\eqno{(2.3^\prime)}
$$
where $\frac{\partial^\alpha u(x,t)}{\partial t^\alpha}$ is the time
fractional Caputo derivative defined as
%
\begin{equation}
\frac{\partial^\alpha u(x,t)}{\partial
t^\alpha}=\frac{1}{\Gamma(1-\alpha)}\int^{t}_{0}\frac{\partial
u(x,s)}{\partial s}\frac{1}{(t-s)^{\alpha}}ds,\ 0<\alpha <1.
\end{equation}
For ease of presentation, we uniformly divide the spacial domain
$\Omega=(0,1)$ into $M$ subintervals with stepsize $h$ and the time
domain $(0,T)$ into $N$ subintervals with steplength $\tau$.
Let $x_{i}=i h,\ i=0,1,\cdots,M$; $t_{j}=j\tau,\ j=0,1,\cdots,N$.
Let the grid function be ${u^{j}_{i}=u(x_i,t_j),\, 0\leq i\leq
M,\,0\leq j\leq N}$, denote
\begin{equation} \label{2.5}
\delta_{x}^2u^{j}_{i}=\frac{u^{j}_{i+1}-2u^{j}_{i}+u^{j}_{i-1}}{h^2},
\end{equation}
and define
\begin{equation}\label{2.6}
D^{\alpha}_{\tau}u_i^{n}=\frac{\tau^{-\alpha}}{\Gamma(2-\alpha)}\big[u_i^{n}-\sum^{n-1}_{j=1}(b_{n-j-1}-b_{n-j})u_i^{j}-b_{n-1}u_i^{0}\big]
\end{equation}
as the discrete time fractional derivative\cite{Lin2007},  where
$b_{j}=(j+1)^{1-\alpha}-j^{1-\alpha}$. It can be noted that $b_j>0$
and $1=b_0>b_1>\cdots>b_n >
(1-\alpha)(n+1)^{-\alpha}$.  There exists the following
error estimate between $(\partial^\alpha u(x_i,t)/\partial
t^\alpha)|_{t=t_n}$ and $D^{\alpha}_{\tau}u_i^{n}$:
\begin{lemma}[Ref. \refcite{Zhang2011}] \label{lemma2.1}
Suppose $0<\alpha<1$, and let $u(x_i, t)\in C^{2}[0,t_n]$, then
$$
\left| \frac{\partial^\alpha u(x_i,t)}{\partial t^\alpha}|_{t=t_n}
-D^{\alpha}_{\tau} u_i^n \right| \le \frac{6}{\Gamma(2-\alpha)}\cdot
\max\limits_{0 \le t \le t_n} |\partial^2 u(x_i,t)/\partial
t^2|\cdot \tau^{2-\alpha}.
$$
\end{lemma}
\noindent
And it is well known that $\delta_{x}^2u^{j}_{i}$ is the
2nd order central difference approximation of $\partial^2
u(x,t)/\partial x^2$ at $(x_i,t_j)$. Replacing $n$ by $n+1$,  Eq.
(\ref{2.6}) can also be recast as
\begin{equation} \label{2.7}
\begin{array}{ll}
D^{\alpha}_{\tau}u^{n+1}
&=\displaystyle\frac{\tau^{1-\alpha}}{\Gamma(2-\alpha)}\sum^{n}_{j=0}b_{j}\frac{u^{n-j+1}-u^{n-j}}{\tau}\\\\
&=\displaystyle\frac{1}{\Gamma(1-\alpha)\tau^{\alpha}}[u^{n+1}-\sum^{n-1}_{j=0}(b_{j}-b_{j+1})u^{n-j}-b_{n}u^{0}].
\end{array}
\end{equation}
Combining (\ref{2.5}) and (\ref{2.7}), we
design the semi-implicit finite difference scheme of (\ref{2.1}) as
\begin{equation}
D^{\alpha}_{\tau}U_i^{n+1}=\delta_x^2 U^{n+1}_{i}+f(U^{n}_{i}),\label{2.8}
\end{equation}
where $i=1,2,\cdots,M-1$ for boundary condition (\ref{2.13}), and $i=0,1,\cdots,M$ for ($2.13^\prime$), $n=1,2,\cdots,N-1$. Denoting $\Gamma(2-\alpha)\tau^{\alpha}$ by $C_{\alpha}$
we rewrite the
above semi-implicit finite difference scheme as
\begin{equation}
\sum^{n}_{j=0}b_{j}\frac{u^{n-j+1}-u^{n-j}}{\tau}=C_{\alpha}\delta_x^2 U^{n+1}_{i}+C_{\alpha}f(U^{n}_{i}).\label{2.10}
\end{equation}
From (\ref{2.2}) the initial
condition is specified as
\begin{equation}\label{2.12}
U^{0}_{i}=g(x_{i}), ~~ {\rm for }~~ i=0,1,\cdots,M;
\end{equation}
and from (\ref{2.3}) or ($2.3^\prime$) the boundary conditions are given as
\begin{equation}\label{2.13}
 U^{n}_{0}=U^{n}_{M}=0, ~~ {\rm for }~~ n=0,1,\cdots,N,
\end{equation}
or
$$
 U^{n}_{-1}=U^{n}_{1},\,\,\,\,  U^{n}_{M+1}=U^{n}_{M-1}, ~~ {\rm for }~~ n=0,1,\cdots,N, \eqno(2.13^\prime)
$$
namely, the central difference discretization is used for the Neumann boundary.



\section{Stability and convergence of the numerical scheme (\ref{2.10})-(\ref{2.13})}
Now we discuss the stability and convergence of the numerical schemes, first we analyze the numerical stability.
Let $\widetilde{U}_i^n$ be the approximate solution of the numerical
scheme (\ref{2.10})-(\ref{2.13}),
and denote $\epsilon_i^n= U^n_i -\widetilde{U}^n_i$. From
(\ref{2.10}), we immediately obtain
\begin{equation} \label{3.1}
\epsilon_i^{n+1}=\sum\limits_{j=0}^n b_j
\epsilon_i^{n-j}-\sum\limits_{j=1}^n b_j
\epsilon_i^{n+1-j}+\frac{C_\alpha}{h^2}(\epsilon_{i+1}^{n+1}-2\epsilon_i^{n+1}+\epsilon_{i-1}^{n+1})+C_\alpha
\big( f(U_i^n)-f(\widetilde{U}_i^n) \big),
\end{equation}
while the perturbation errors of boundary conditions are
\begin{equation} \label{3.01}
\displaystyle \epsilon_0^n=\epsilon_M^n=0,\quad 1\leq n\leq N.
\end{equation}
Throughout the paper, we assume that the function $f$ satisfies the local
Lipschitz condition with the Lipschitz constant $L$, namely,
$$|f(u(x,t))-f(\widetilde{u}(x,t))|\leq L|u(x,t)-\widetilde{u}(x,t)|,$$
when $ |u(x,t)-\widetilde{u}(x,t)|<\varepsilon_0,~~ x \in \Omega ~~{\rm and }~~  t\in [0,T],$
where $\varepsilon_0$ is a given positive constant.

Denote
$e^{n}=(\epsilon_{1}^{n},\epsilon_{2}^{n},\cdots,\epsilon_{M-1}^{n})^{T}$.
We define the discrete $L^2$ norm as $\|e^n\|=\left(h
\sum\limits_{i=0}^{M} \epsilon_i^n \right)^{1/2}$.
Numerical stability result is as follows.

\begin{theorem}\label{Th3.1}
The numerical scheme (\ref{2.10})-(\ref{2.13}) is stable and there
exists
\begin{equation}\label{3.2}
\| e ^n \|\leq\frac{1}{(1-\alpha)-T^{\alpha}\Gamma(2-\alpha)L}\| e^0
\|,
\end{equation}
when $T < (1/(\Gamma(1-\alpha)L))^{1/\alpha} $.
\end{theorem}

\begin{proof}
We prove this theorem by mathematical induction. Taking $n=0$ in (\ref{3.1}), we
have
\begin{equation} \label{3.3}
\left(1+ \frac{2C_\alpha}{h^2} \right) \epsilon_i^1=
\epsilon_i^0+\frac{C_\alpha}{h^2}(\epsilon_{i+1}^1+\epsilon_{i-1}^1)+C_\alpha
\big( f(U_i^0)-f(\widetilde{U}_i^0) \big).
\end{equation}
Multiplying both sides of (\ref{3.3}) by $h\epsilon^{1}_i$ and
summing up, there exists
$$
\begin{array}{l}
\displaystyle\left(1+ \frac{2C_\alpha}{h^2} \right) \|e^1\|^2
\\
\\
\displaystyle=\sum\limits_{i=1}^{M-1}
(\epsilon_i^0\cdot\epsilon_i^1\cdot
h)+\frac{C_\alpha}{h^2}\sum\limits_{i=1}^{M-1}((\epsilon_{i+1}^1+\epsilon_{i-1}^1)\cdot\epsilon_i^1\cdot
h)+C_\alpha \sum\limits_{i=1}^{M-1} \big(
f(U_i^0)-f(\widetilde{U}_i^0) \big) \epsilon_i^1h
\\
\\
\displaystyle \le
\|e^0\|\cdot\|e^1\|+\frac{2C_\alpha}{h^2}\|e^1\|^2+C_\alpha \cdot L
\cdot \|e^0\|\cdot\|e^1\|.
\end{array}
$$
Then we obtain
$$
\|e^1\| \le (1+C_\alpha L) \|e^0\|,
$$
it can be easily checked that this means (\ref{3.2}) holds for
$e^1$. Now supposing (\ref{3.2}) holds for $e^1$, $e^2$, $\cdots$,
$e^n$, we prove
\begin{equation*}
\| e ^{n+1}
\|\leq\frac{1}{(1-\alpha)-T^{\alpha}\Gamma(2-\alpha)L}\|e ^0
\|.
\end{equation*}
Just as the above process, multiplying both sides of (\ref{3.1}) by
$h\epsilon^{n+1}_i$ and summing up, we get
$$
\begin{array}{l}
\displaystyle\left(1+\frac{2C_\alpha}{h^2} \right) \|e^{n+1}\|^2
\\
\\
\displaystyle=
\sum\limits_{j=0}^n\sum\limits_{i=1}^{M-1}b_j\epsilon_i^{n-j}
\epsilon_i^{n+1}
h-\sum\limits_{j=1}^n\sum\limits_{i=1}^{M-1}b_j\epsilon_i^{n+1-j}
\epsilon_i^{n+1} h
\\
\\
\displaystyle ~~~~
+\frac{C_\alpha}{h^2}\sum\limits_{i=1}^{M-1}(\epsilon_{i+1}^{n+1}+\epsilon_{i-1}^{n+1})\epsilon_i^{n+1}
h+C_\alpha \sum\limits_{i=1}^{M-1} \big(
f(U_i^n)-f(\widetilde{U}_i^n) \big) \epsilon_i^{n+1}h
\\
\\
\displaystyle=
\sum\limits_{j=0}^{n-1}\sum\limits_{i=1}^{M-1}(b_j-b_{j+1})\epsilon_i^{n-j}
\epsilon_i^{n+1} h+\sum\limits_{i=1}^{M-1}b_n\epsilon_i^0
\epsilon_i^{n+1} h
\\
\\
\displaystyle ~~~~
+\frac{C_\alpha}{h^2}\sum\limits_{i=1}^{M-1}(\epsilon_{i+1}^{n+1}+\epsilon_{i-1}^{n+1})\epsilon_i^{n+1}
h+C_\alpha \sum\limits_{i=1}^{M-1} \big(
f(U_i^n)-f(\widetilde{U}_i^n) \big) \epsilon_i^{n+1}h
\\
\\
\displaystyle \le \sum\limits_{j=0}^{n-1}
(b_j-b_{j+1})\|e^{n-j}\|\cdot\|e^{n+1}\|+
b_n\|e^0\|\cdot\|e^{n+1}\|+\frac{2C_\alpha}{h^2}\|e^{n+1}\|^2
\\
\\
\displaystyle~~~~ +C_\alpha \cdot L \cdot \|e^n\|\cdot\|e^{n+1}\|.
\end{array}
$$
Then, there exists
$$
\begin{array}{lll}
 \|e^{n+1}\| &\le & \displaystyle\left(\frac{1-b_n+C_\alpha
L}{(1-\alpha)-T^{\alpha}\Gamma(2-\alpha)L} +b_n \right) \|e^0 \| \\
\\
&\le &  \displaystyle
\frac{1}{(1-\alpha)-T^{\alpha}\Gamma(2-\alpha)L}\|e^0 \|.
\end{array}
$$
\end{proof}

\begin{theorem}
Let $u(x_{i},t_n)$ and $U_{i}^n$ be the exact solutions of the
subdiffusive reaction-diffusion equation (\ref{2.1})-(\ref{2.3})  and of the
numerical scheme (\ref{2.10})-(\ref{2.13}), respectively,
 define
$\varepsilon_{i}^{n}=u(x_{i},t_n)-U_{i}^n$ and
$E^{n}=(\varepsilon_{1}^{n},\varepsilon_{2}^{n},\cdots,\varepsilon_{M-1}^{n})^{T}$,
then $E^{n}$satisfies the following error estimate:
\begin{equation}\label{3.5}
\|E^{n}\|\leq C(\tau+h^2),
\end{equation}
when $T < (1/(\Gamma(1-\alpha)L))^{1/\alpha} $.
\end{theorem}

\begin{proof}
According to equation (\ref{2.8}), we know
\begin{equation}
D^{\alpha}_{\tau}U_i^{n+1}=\delta_x^2 U_i^{n+1}+f(U_i^{n})\label{3.6}.
\end{equation}
By Lemma $2.1$, there exists a positive constant $C$, such that
\begin{equation}
\left|\frac{\partial^\alpha u(x_i,t)}{\partial t^\alpha}|_{t=t_{n+1}}
-D^{\alpha}_{\tau} u_i^{n+1} \right|\leq C\tau ^{2-\alpha},\label{r1}
\end{equation}
and it is well known that
\begin{equation}
\left|\frac{\partial^2 u(x_i,t_{n+1})}{\partial x^2}-\delta_x^2 u_i^{n+1}\right|\leq Ch^2.\label{r2}
\end{equation}
Thanks to the Lipschitz continuity of $f$ with respect to $u$, we have
\begin{equation}
\left|f(u_i^{n+1})-f (u_i^{n}) \right|\leq L\left|u_i^{n+1}-u_i^{n}\right| \leq C\tau.\label{r3}
\end{equation}
Based on (\ref{r1})-(\ref{r3}), there exists
$$
D^{\alpha}_{\tau}u_i^{n+1}=\delta_x^2 u^{n+1}_{i}+f(u^{n}_{i})+R_i^{n+1},
$$
where $R_i^{n}=O(\tau+h^2)$.
For convenience, denoting $C_\alpha R_i^{n+1}$ as $r_i^{n+1}$, we get
\begin{equation}\label{3.10}
\begin{array}{lll}
\displaystyle\left(1+ \frac{2C_\alpha}{h^2} \right)\varepsilon_i^{n+1}= & \displaystyle \sum\limits_{j=0}^n b_j
\varepsilon_i^{n-j}-\sum\limits_{j=1}^n b_j\varepsilon_i^{n+1-j}
+\frac{C_\alpha}{h^2}\left(\varepsilon_{i+1}^{n+1}+\varepsilon_{i-1}^{n+1}\right)
\\
\\
\displaystyle~~~~&+C_\alpha\big( f(u(x_{i},t_{n}))-f(U_{i}^{n}) \big)+ r_i^{n+1}.
\end{array}
\end{equation}
%
Similar to the proof of Theorem \ref{Th3.1}, taking $n=0$ in (\ref{3.10}), multiplying both sides of (\ref{3.10}) by $h\varepsilon^{1}_i$,
summing up and directly calculating, there exists
\begin{align}
||E^{1}||^2\leq b_0 ||E^{0}||\cdot||E^1||+C_{\alpha} L||E^{0}||\cdot||E^1||+||r^{1}||\cdot||E^1||, \nonumber
\end{align}
where $r^{1}=(r^{1}_1,\cdots,r^{1}_{M-1})$.
Then 
$$
\begin{array}{lll}
||E^{1}|| &\leq &\displaystyle(1+C_{\alpha} L)|| E^{0}||+|| r^{1}||
\\
\\
&\leq &\displaystyle\frac{b_0^{-1}}{(1-\alpha)-T^{\alpha}\Gamma(2-\alpha)L}(|| E^{0}||+|| r^{1}||).
\end{array}
$$
Now suppose that
\begin{equation}\label{t3.2}
 \|E^{m}\| \le  \displaystyle
 \frac{b_{m-1}^{-1}}{(1-\alpha)-T^{\alpha}\Gamma(2-\alpha)L}(\|E^0 \|+\max_{1\leq j\leq m}|| r^{j}||),
\end{equation}
holds for $m=1,2,\cdots,n$, then we prove that it holds for $E^{n+1}$. Similarly,
$$
\begin{array}{l}
\displaystyle\left(1+ \frac{2C_\alpha}{h^2} \right)||E^{n+1}||^2 \\
\\
\displaystyle
=\sum_{i=1}^{M-1}\sum_{j=0}^n b_j\varepsilon^{n-j}_i\varepsilon^{n+1}_ih-\sum_{i=1}^{M-1}\sum_{j=1}^n b_j\varepsilon^{n+1-j}_i \varepsilon^{n+1}_ih
\\
\\
\displaystyle~~~ +\sum_{i=1}^{M-1}\frac{C_\alpha}{h^2}(\varepsilon_{i+1}^{n+1}+\varepsilon_{i-1}^{n+1})\varepsilon_i^{n+1}
h+ C_{\alpha}\sum_{i=1}^{M-1} \left( f\left( u(x_{i},t_{n}) \right) -f( U_{i}^{n}) \right)\varepsilon^{n+1}_i h
\\
\\
\displaystyle~~~
+\sum_{i=1}^{M-1}r^{n+1}_i \varepsilon^{n+1}_ih
 \\
\\
\displaystyle\leq \sum_{j=0}^{n-1}(b_j-b_{j+1})(||E^{n-j}|| \cdot||E^{n+1}||) +b_n ||E^{n+1}||\cdot ||E^0||
\\
\\
\displaystyle ~~~~
+C_{\alpha} L||E^n||\cdot||E^{n+1}||+ ||r^{n+1}||\cdot||E^{n+1}||.
\end{array}
$$
Dividing by $||E^{n+1}||$ at both sides, we conclude that
$$
\begin{array}{l}
\displaystyle||E^{n+1}||\displaystyle
 \leq \sum_{j=0}^{n-1}(b_j-b_{j+1})||E^{n-j}||  +b_n  ||E^0||+C_{\alpha} L||E^{n}||+||r^{n+1}||.
\end{array}
$$
According to (\ref{t3.2}), combining with $b_{j}^{-1}<b_{j+1}^{-1}$, we have
$$
\begin{array}{l}
\displaystyle||E^{n+1}||
\\
\\
\displaystyle\leq
\sum_{j=0}^{n-1}\frac{(b_j-b_{j+1})b_{n-j-1}^{-1}}{(1-\alpha)-T^{\alpha}\Gamma(2-\alpha)L}(\|E^0 \|+\max_{1\leq j\leq {n-1}}|| r^{j}||)
\\
\\
\displaystyle~~~~
+b_n  ||E^0||+\frac{C_{\alpha} Lb_{n-1}^{-1}}{(1-\alpha)-T^{\alpha}\Gamma(2-\alpha)L}(\|E^0 \|+\max_{1\leq j\leq {n}}|| r^{j}||)+|| r^{n+1}||
\\
\\
\displaystyle\leq
\frac{b_{n}^{-1}}{(1-\alpha)-T^{\alpha}\Gamma(2-\alpha)L}(\|E^0 \|+  \max_{1\leq j\leq {n+1}}|| r^{j}||).
\end{array}
$$
Notice that
$$
E^0_i=0, ~E_0^j=E_M^j=0,~~~~0\leq i\leq M,1\leq j\leq n+1.
$$
Together with
$b_n^{-1}\leq\frac{(n+1)^{\alpha}}{1-\alpha}$ and $(n+1)\tau\leq T$, we
obtain
$$
\begin{array}{lll}
||E^{n+1}|| &\leq&  \displaystyle\frac{b_{n}^{-1}}{(1-\alpha)-T^{\alpha}\Gamma(2-\alpha)L}    \max\limits_{1\leq j\leq {n+1}}|| r^j||
\\
\\
&\leq&\displaystyle  \frac{b_{n}^{-1}C_{\alpha} }{(1-\alpha)-T^{\alpha}\Gamma(2-\alpha)L}   \max_{1\leq j\leq {n+1}}|| R^{j}||
\\
\\
&\leq& \displaystyle \frac{T^\alpha \Gamma(1-\alpha) }{(1-\alpha)-T^{\alpha}\Gamma(2-\alpha)L}   \max_{1\leq j\leq {n+1}}|| R^{j}||
\\
\\
&\leq& \displaystyle C(\tau+h^2) .
\end{array}
$$
\end{proof}

\begin{remark} The above analysis focuses on scalar equation, but it can be easily extended to the vector one. Take (\ref{Eq1.1}) as an example, and denote the nonlinear term of the first equation as $f(N,P)$ and the nonlinear term of the second equation as $g(N,P)$.
Denoting $\rho^n=N(x_i,t_n)-N_i^n$ and $\eta^n=P(x_i,t_n)-P_i^n$, and using the following simple trick
$$
\begin{array}{l}
\displaystyle
|f(N_i^{n+1},P_i^{n+1})-f(N_i^{n},P_i^{n})|\\\\
\displaystyle =|f(N_i^{n+1},P_i^{n+1})-f(N_i^{n+1},P_i^{n})+f(N_i^{n+1},P_i^{n})-f(N_i^{n},P_i^{n})|\\\\
\displaystyle \leq|f(N_i^{n+1},P_i^{n+1})-f(N_i^{n+1},P_i^{n})|+|f(N_i^{n+1},P_i^{n})-f(N_i^{n},P_i^{n})|\\\\
\displaystyle \leq L(|P_i^{n+1}-P_i^{n}|+|N_i^{n+1}-N_i^{n}|),
\end{array}
$$
lead to 
$$
 \|\rho^{m}\| \le  \displaystyle
 \frac{b_{m-1}^{-1}}{(1-\alpha)-T^{\alpha}\Gamma(2-\alpha)L} (\|\rho^0\|+\| \eta^0\|+\max_{1\leq j\leq m}|| r^{j}||),
$$
and
$$
 \|\eta^{m}\| \le  \displaystyle
  \frac{b_{m-1}^{-1}}{(1-\alpha)-T^{\alpha}\Gamma(2-\alpha)L} (\|\rho^0\|+\| \eta^0\|+\max_{1\leq j\leq m}|| r^{j}||),
$$
for $m=1,2,\cdots,n$, by the analysis similar to Theorem $3.2$.

\end{remark}

\begin{remark}
The above analysis is for $L^2$ estimates, and we hope that the $H^1$ estimates can be obtained in the near future, but it needs more delicate tricks to dealing with the nonlinear terms. In fact, there are already the $H^1$ estimates for the linear equations \cite{Gao2011}.
\end{remark}

\section{Positivity and boundedness of the analytical and numerical solutions of the subdiffusive predator-prey model}
 In this section, we first prove that the analytical solutions of (\ref{Eq1.1})-(\ref{Eq1.2}) are positive and bounded, then demonstrate that both the numerical schemes (\ref{2.10})-(\ref{2.12}) with (\ref{2.13}) and with ($2.13^\prime$) preserve their positivity and boundedness when utilized to numerically solve (\ref{Eq1.1})-(\ref{Eq1.2}) and (\ref{Eq1.1}) and ($1.2^\prime$), respectively.

\subsection{Maximum principle for analytical solutions}
 For analyzing the properties of the analytical solutions, we introduce the following maximum principle.
The considered equation is
 \begin{equation}\label{4.1}
  Lu=\frac{\partial^\alpha u}{\partial t^\alpha}-\frac{\partial^2 u}{\partial x^2}+c(x,t)u=f(x,t), ~~~~ (x,t) \in \Omega_T=\Omega \times (0,T],
 \end{equation}
 where $\Omega_T$ is a bounded domain with Lipschitz continuous boundary.

\begin{theorem} \label{Th4.1}
Assuming that the coefficient $c(x,t)\ge 0$ and $f(x,t)\leq 0$ (resp. $f(x,t)\geq 0$) in $\Omega_T$ and $u\in C^{2,1}(\Omega_T)\bigcap C(\overline{\Omega}_T)$ is the solution of (\ref{4.1}), then the non-negative maximum (resp. non-positive minimum) value of $u(x,t)$ in $\Omega_T$ (if exists) must reach at the parabolic boundary $\Gamma_T$, namely
\begin{equation} \label{4.2}
\max_{\overline{\Omega}_T}u(x,t)\leq \max_{\Gamma{_T}} \lbrace u(x,t),0\rbrace\,\,\,\,\,\, ({\rm resp.}\,\,\min_{\overline{\Omega}_T}u(x,t)\geq \min_{\Gamma{_T}} \lbrace u(x,t),0\rbrace).
\end{equation}
\end{theorem}
In fact, if the non-negative maximum value of $u(x,t)$ is not at the boundary $\Gamma{_T}$ and $f(x,t) \leq 0$, then there exists a point $(x^*,t^*)\in \Omega_T$ such that
\[u(x^*,t^*)> \max_{\Gamma{ _T}}\lbrace u(x,t),0\rbrace~~{\rm and }~~ u(x^*,t^*) \ge \max_{\overline{\Omega}_T} u(x,t).\]
Let $b>0$, for any $\varepsilon >0$, we introduce the auxiliary function
\[v(x,t)=u(x,t)-\varepsilon e^{bt}.\]
On the one hand, we know that, for any $(x,t) \in \Omega_T$, $v(x,t)$ satisfies
\begin{equation}\label{t12}
\begin{array}{lll}
\displaystyle
\frac{\partial^{\alpha}v}{\partial t^{\alpha}}-\frac{\partial^2 v}{\partial x^2}+c(x,t)v
& =& \displaystyle\frac{\partial^{\alpha}u}{\partial t^{\alpha}}-\frac{\partial^2 u}{\partial x^2}+c(x,t)u-\varepsilon e^{bt}c(x,t)\\
\\
&=& f(x,t)-\varepsilon (bt^{1-\alpha}E_{1,2-\alpha}(bt)+ e^{bt}c(x,t))<0,
\end{array}
\end{equation}
where $E_{\alpha,\beta}(z)$ is the Mittag-Leffler function. 
At the maximum point $(x^*,t^*)$, according to the definition of Caputo derivative, we have
\begin{align}
\frac{\partial^{\alpha}u(x^*,t)}{\partial t^{\alpha}}|_{t=t^*}=&\frac{1}{\Gamma(1-\alpha)} \int^{t^*}_0 \frac{\partial u(x^*,s)}{\partial s}\frac{1}{(t^* - s)^{\alpha}}ds \nonumber\\
 =&\lim_{\tau\rightarrow 0}\frac{\tau^{1-\alpha}}{\Gamma(2-\alpha)}\sum^{n-1}_{j=0}b_j\frac{1}{\tau}(u(x^*,t^*-j\tau)-u(x^*,t^*-(j+1)\tau))  ,\nonumber \\
 =&\lim_{\tau\rightarrow 0}\frac{\tau^{-\alpha}}{\Gamma(2-\alpha)}(u(x^*,t^*)-(1-b_1)u(x^*,t^*-\tau)-\cdots-b_nu(x'_0,0))\nonumber\\
=& \lim_{\tau\rightarrow 0}\frac{\tau^{-\alpha}}{\Gamma(2-\alpha)}((1-b_1)(u(x^*,t^*)-u(x^*,t^*-\tau))+\cdots\nonumber\\
&+b_n(u(x^*,t^*)-u(x^*,0)))\nonumber\\
\geq& \lim_{\tau\rightarrow 0}\frac{\tau^{-\alpha}}{\Gamma(2-\alpha)}b_n  (u(x^*,t^*)-u(x^*,0) ) \nonumber\\
>& \lim_{\tau\rightarrow 0}\frac{(1-\alpha)(n+1)^{-\alpha}\tau^{-\alpha}}{\Gamma(2-\alpha)}  (u(x^*,t^*)-u(x^*,0)) \nonumber\\
\geq& \frac{(1-\alpha)T^{-\alpha}}{\Gamma(2-\alpha)} (u(x^*,t^*)-u(x^*,0)), \nonumber
\end{align}
where $\tau=t^*/n,\,b_j$ is defined in $(\ref{2.6})$, and $1=b_0>b_1>\cdots>b_n >
(1-\alpha)(n+1)^{-\alpha}$ is used. Since  $u(x^*,t^*)> u(x^*,0)$, denoting $m^*=u(x^*,t^*)-u(x^*,0)$, there exists
$$
\displaystyle\frac{\partial^{\alpha}v(x^*,t)}{\partial t^{\alpha}}|_{t=t^*}=\frac{\partial^{\alpha}u(x^*,t^*)}{\partial t^{\alpha}}
-\varepsilon b(t^*)^{1-\alpha}E_{1,2-\alpha}(bt^*)\geq0,
$$
when $\varepsilon\leq \frac{(1-\alpha)T^{-\alpha}m^*}{\Gamma(2-\alpha)b(t^*)^{1-\alpha}E_{1,2-\alpha}(bt^*)}$.
Together with
$\frac{\partial^2 v}{\partial x^2}=\frac{\partial^2 u}{\partial x^2} \le 0$ at $(x^*,t^*)$, 
  we know
\[ \frac{\partial^{\alpha}v}{\partial t^{\alpha}}-\frac{\partial^2 v}{\partial x^2}+c(x,t)v\geq 0 \,\,\,\,\,\, at ~~~ (x^*,t^*),\]
which is contradictory with (\ref{t12}).
Similar analysis can be done for the case $f(x,t) \geq 0$. So, from the above analysis, we arrive at Theorem \ref{Th4.1}.



\begin{remark}\label{Remark4.1}
If we take the initial condition of (\ref{4.1}) as $u(x,0)=0$, and its boundary conditions are Dirichlet's and homogeneous, then from Theorem \ref{Th4.1} we have $u(x,t) \leq 0$ when $f(x,t) \leq 0$ and $u(x,t) \geq 0$ when $f(x,t) \geq 0$. The same results still hold if the homogeneous Neummann boundary conditions are used, since the maximum (minimum) value of $u(x,t)$ at the boundary is non-positive (nonnegative) under the homogeneous Neummann boundary conditions. In fact, if the maximum value of $u(x,t)$ at the boundary is positive, suppose it is located at the left boundary (similar analysis can be done if at the right boundary) and denote the one closest to the line $t=0$ by $u(x_l,t^*)$, then for any given sufficiently small $\varepsilon$, there exists $\xi_\varepsilon \in (x_l,x_l+\varepsilon)$ such that $u(x_l+\varepsilon,t^*)-u(x_l,t^*)=\big(\frac{\partial^2 u(x, t^*)}{\partial x^2}|_{x=\xi_\varepsilon}\big) \varepsilon^2<0$ since $\frac{\partial u(x, t^*)}{\partial x}|_{x=x_l}=0$.  So there exists sufficiently big $M>2$ and small $\delta t>0$  such that $\frac{\partial^2 u(x, t)}{\partial x^2}<0$ for any $(x,t) \in \Omega_\varepsilon^*=\{(x,t) \,| x_l<\xi_\varepsilon-\frac{\varepsilon}{M}<x<\xi_\varepsilon+\frac{\varepsilon}{M}<x_l+\varepsilon;\,  t^*-\delta t<t<t^* \}$. Now consider (\ref{4.1}) in the domain $\Omega^*=\{(x,t)\,| \xi_\varepsilon<x<x_l+\varepsilon;\, 0<t<t^*\}$, obviously the maximum value of $u(x,t)$ still is obtained at the parabolic boundary $\Gamma_{t^*}$ of $\Omega^*$. Furthermore, if taking $\varepsilon$ small enough, then the maximum value is located in the domain $\Gamma_{t^*} \cap \Omega_\varepsilon^*$. Now at the maximum point, $\frac{\partial^{\alpha}u }{\partial t^{\alpha}}-\frac{\partial^2 u }{\partial x^2}+c u > 0$, a contradiction is reached.

\end{remark}

\subsection{Positiveness and boundedness of the analytical solutions}
Using the upper and lower solutions method, we prove the positiveness and boundedness of the analytical solutions (\ref{Eq1.1})-(\ref{Eq1.2}) and (\ref{Eq1.1}) and ($1.2^\prime$). First, we introduce the definition of the upper and lower solutions.

\begin{definition}
For the system of equations ($i=1,2$)
\begin{align}
\frac{\partial^{\alpha}u_i}{\partial t^{\alpha}}-\frac{\partial^2 u_i}{\partial x^2}&= f_i(u_1,u_2),\ \,x\in\Omega,\,t\in(0,T],\nonumber\\
Bu_i\Big(\,u_i\, {\rm or} \, \frac{\partial u_i}{\partial x}\Big) &= g_i(x,t),\quad\ \,x\in\partial\Omega,\,t\in(0,T],\label{4.3} \\
u_i(x,0)& =\varphi_i(x),\qquad x\in\Omega,\nonumber
\end{align}
suppose that $\tilde{u}_i(x,t)$ and $\utilde{u}_i(x,t)$ satisfy
\begin{align}
&B \tilde{u}_i  -g_i(x,t)\geq 0\geq   B\utilde{u}_i -g_i(x,t),\ \,x\in\partial\Omega,\,t\in(0,T],\label{4.4.1}\\
&\tilde{u}_i(x,0)-\varphi_i(x)\geq 0 \geq\utilde{u}_i(x,0)-\varphi_i(x),\quad \ x\in\bar{\Omega}, \label{4.4.2}
\end{align}
and $f_1(\cdot,\cdot)$ is quasi-monotone decreasing, $f_2(\cdot,\cdot)$ is quasi-monotone increasing, and
\begin{align}
\frac{\partial^{\alpha}\tilde{u}_1}{\partial t^{\alpha}}&-\frac{\partial^2 \tilde{u}_1}{\partial x^2}-f_1(\tilde{u}_1,\utilde{u}_2)\geq 0\geq \frac{\partial^{\alpha}\utilde{u}_1}{\partial t^{\alpha}}-\frac{\partial^2 \utilde{u}_1}{\partial x^2}-f_1(\utilde{u}_1,\tilde{u}_2),\label{4.4.3}\\
\frac{\partial^{\alpha}\tilde{u}_2}{\partial t^{\alpha}}&-\frac{\partial^2 \tilde{u}_2}{\partial x^2}-f_2(\tilde{u}_1,\tilde{u}_2)\geq 0\geq \frac{\partial^{\alpha}\utilde{u}_2}{\partial t^{\alpha}}-\frac{\partial^2 \utilde{u}_2}{\partial x^2}-f_2(\utilde{u}_1,\utilde{u}_2),\label{4.4.4}
\end{align}
then $U(x,t)=(\tilde{u}_1(x,t),\tilde{u}_2(x,t))$ and $V(x,t)=(\utilde{u}_1(x,t),\utilde{u}_2(x,t))$ are respectively called upper solution and lower solution of the system (\ref{4.3}).
\end{definition}


\begin{theorem}\label{Th4.2}
Suppose $\{f_1,f_2\}$ is mixed quasi-monotonous and Lipschitz continuous with respect to $u_1$ and $u_2$
$$
|f_i(u_1,u_2)-f_i(v_1,v_2)|\leq L(|u_1-v_1|+|u_2-v_2|),
$$
where $L$ is constant.
 If the upper and lower solutions,
$U(x,t)$ and $V(x,t)$, satisfy $V(x,t)\leq U(x,t)$, then (\ref{4.3}) has a unique solution in $[V(x,t), U(x,t)]$.
\end{theorem}
\begin{proof}
See Appendix A.
\end{proof}

Let us denote $f(N,P)=N(1-N-\frac{a P}{P+N})$ and $g(N,P)=\sigma P(-\frac{\gamma+\delta \beta P}{1+\beta P}+\frac{N}{P+N})$. We certify that both (\ref{Eq1.1})-(\ref{Eq1.2}) and (\ref{Eq1.1}) and ($1.2^\prime$) have lower and upper solutions, $(0,0)$ and $(1,L_1)$, where $L_1>1/\gamma$. First, it is needed to check that $f$ is quasi-monotone decreasing function and $g$ is quasi-monotone increasing function. Since the derivatives of the nonlinear terms are
\[\frac{\partial f(N,P)}{ \partial P}=-\frac{aN^2}{(P+N)^2}\]
and
\[\frac{\partial g(N,P)}{ \partial N}=\frac{\sigma P^2}{(P+N)^2},\]
 it is obvious that $f$ is decreasing w.r.t $P$ and $g$ is increasing w.r.t $N$.
At the same time, it can be noted that both (\ref{Eq1.2}) and ($1.2^\prime$)  satisfy the conditions (\ref{4.4.1}) directly.
And it can be easily checked that the upper solution $U(x,t)=(\tilde{N}(x,t),\tilde{P}(x,t))=(1,L_1)$
 and the lower solution $V(x,t)=(\utilde{N}(x,t),\utilde{P}(x,t))=(0,0)$ satisfy
\begin{align}
\frac{\partial^{\alpha}\tilde{N}}{\partial t^{\alpha}}&-\frac{\partial^2 \tilde{N}}{\partial x^2}-f(\tilde{N},\utilde{P})
\geq 0\geq \frac{\partial^{\alpha}\utilde{N}}{\partial t^{\alpha}}-\frac{\partial^2 \utilde{N}}{\partial x^2}-f(\utilde{N},\tilde{P}),\\
\frac{\partial^{\alpha}\tilde{P}}{\partial
t^{\alpha}}&-\frac{\partial^2 \tilde{P}}{\partial
x^2}-g(\tilde{N},\tilde{P})\geq 0\geq
\frac{\partial^{\alpha}\utilde{P}}{\partial
t^{\alpha}}-\frac{\partial^2 \utilde{P}}{\partial
x^2}-g(\utilde{N},\utilde{P}).
\end{align}
So if we specify the initial condition of (\ref{Eq1.1}) such that $N(x,0) \in [0,1]$ and $P(x,0) \in [0,L_1]$ for any $x \in (l,r)$, then the initial condition satisfy (\ref{4.4.2}). From Theorem \ref{Th4.2} the exact solution of (\ref{Eq1.1})-(\ref{Eq1.2}) (or (\ref{Eq1.1}) and ($1.2^\prime$)) is bounded and positive.

\subsection{Positiveness and boundedness of the numerical solutions}
We show that the numerical schemes (\ref{2.10})-($2.13^\prime$) preserve the positiveness and boundedness of the corresponding analytical solutions of (\ref{Eq1.1})-($1.2^\prime$).

First, for (\ref{Eq1.1}) with the initial conditions $N(x,0) \in (0,1]$ and $P(x,0) \in (0,L_1]$, $L_1<\frac{1}{\gamma}$, for any $x \in (l,r)$, we have its discretization scheme
$$
\begin{array}{llll}
\displaystyle
N^{n+1}_i-C_{\alpha}\delta_x^2 N^{n+1}_{i}=\sum\limits^{n-1}_{j=0}(b_{j}-b_{j+1})N_{i}^{n-j}+b_{n}N_{i}^{0}
\\
\\
\displaystyle~~~~~~~~~~~~~~~~ ~~~~~~~~~~~~+C_{\alpha}N^{n}_{i}(1-N^{n}_{i}-\frac{a P^{n}_{i}}{P^{n}_{i}+N^{n}_{i}}),
\\
\\
\displaystyle P^{n+1}_i-C_{\alpha}\delta_x^2 P^{n+1}_{i}=\sum\limits^{n-1}_{j=0}(b_{j}-b_{j+1})P_{i}^{n-j}+b_{n}P_{i}^{0}
\\
\\
\displaystyle~~~~ ~~~~~~~~~~~~~~~~~~~~~~~~+C_{\alpha}\sigma P^{n}_{i}(-\frac{\gamma+\delta \beta P^{n}_{i}}{1+\beta P^{n}_{i}}+\frac{N^{n}_{i}}{P^{n}_{i}+N^{n}_{i}}).
\end{array}
$$
We use the induction method to prove $0<N_i^n\leq 1$ and $0<P_i^n\leq L_1$ for any $i$ and $n$.  First, $0<N_i^0\leq 1$ and $0<P_i^0\leq L_1$ hold obviously. Now suppose $0<N_i^k\leq 1$ and $0<P_i^k\leq L_1$ for any $ k \leq n$, we prove that it still holds when $k=n+1$.

%
%
First of all, denote $w:=C_{\alpha}N^{n}_{i}(1-N^{n}_{i}-\frac{a
P^{n}_{i}}{P^{n}_{i}+N^{n}_{i}})$ and $v:=C_{\alpha}\sigma
P^{n}_{i}(-\frac{\gamma+\delta \beta P^{n}_{i}}{1+\beta
P^{n}_{i}}+\frac{N^{n}_{i}}{P^{n}_{i}+N^{n}_{i}})$. When
$C_{\alpha}\leq1$, it's easy to obtain
\begin{equation}
w\leq C_{\alpha}N^{n}_{i}(1-N^{n}_{i})\leq
C_{\alpha}(1-N^{n}_{i}),\label{1}
\end{equation}
\begin{equation}
w\geq C_{\alpha}N^{n}_{i}(1-N^{n}_{i}-a)\geq -C_{\alpha}aN^{n}_{i};\label{2}
\end{equation}
and
\begin{equation}
v\leq C_{\alpha}\sigma P^{n}_{i}(-\frac{\gamma+\delta \beta
P^{n}_{i}}{1+\beta P^{n}_{i}}+\frac{1}{P^{n}_{i}})\leq
C_{\alpha}\sigma (1-P^{n}_{i}\gamma),\label{3}
\end{equation}
\begin{equation}
v\geq C_{\alpha}\sigma
P^{n}_{i}(-\delta+\frac{N^{n}_{i}}{P^{n}_{i}+N^{n}_{i}})\geq
-C_{\alpha}\sigma P^{n}_{i}\delta.\label{4}
\end{equation}
Because of $b_{j}> b_{j+1}$, $0<N_i^n\leq 1$ and $0<P_i^n\leq L_1$, we know
\begin{equation}
(1-b_1)N^{n}_{i}\leq\sum\limits^{n-1}_{j=0}(b_{j}-b_{j+1})N_{i}^{n-j}+b_{n}N_{i}^{0}\leq
\frac{1-b_1}{2}N^{n}_{i}+\frac{1+b_1}{2},\label{5}
\end{equation}
and
\begin{equation}
(1-b_1)P^{n}_{i}\leq\sum\limits^{n-1}_{j=0}(b_{j}-b_{j+1})P_{i}^{n-j}+b_{n}P_{i}^{0}\leq
\frac{1-b_1}{2}P^{n}_{i}+\frac{1+b_1}{2}L_1.\label{6}
\end{equation}
Owing to (\ref{1}), (\ref{2}) and (\ref{5}), when
$C_{\alpha}<\min\{\frac{1-b_1}{a},\frac{1-b_1}{2}\} $, we get
\begin{equation}
0< N^{n+1}_i-C_{\alpha}\delta_x^2 N^{n+1}_{i}\leq1.\label{4.15}
\end{equation}
In the similar way, owing to (\ref{3}), (\ref{4}) and (\ref{6}), when
$C_{\alpha}<\frac{(1-b_1)}{\sigma\delta}$, we get
\begin{equation}
0< P^{n+1}_i-C_{\alpha}\delta_x^2 P^{n+1}_{i}\leq L_1.\label{4.16}
\end{equation}
From (\ref{4.15}) and ({\ref{4.16}}), we can get $0<N_i^{n+1}\leq 1$ and $0<P_i^{n+1}\leq L_1$.
In fact, if  $0<N_i^{n+1}\leq 1$  doesn't hold, then there exists $i$ such
that $$N_i^{n+1}\leq 0 \quad {\rm or} \quad N_i^{n+1}>1.$$ If
$N_i^{n+1}\leq 0$, then we choose the minimum in $i=0,\cdots,M$,
and denote it by $N_k^{n+1}$, which is non-positive. Thanks to (\ref{4.15}), we know that
$$N^{n+1}_k-\frac{C_{\alpha}}{h^2}(N^{n+1}_{k-1}-2N^{n+1}_{k}+N^{n+1}_{k+1})>0.$$
So $N_k^{n+1}\leq 0$ implies
$$N^{n+1}_{k}>\frac{N^{n+1}_{k-1}+N^{n+1}_{k+1}}{2}. $$
Then we get $ N^{n+1}_{k-1}< N^{n+1}_{k}$ or
$N^{n+1}_{k+1}<N^{n+1}_{k}$, which still hold even at the boundary, including the Dirichlet and Neummann boundaries.
This is contradictory with the assumption that $N_k^{n+1}$ is the minimum. 
If $N_i^{n+1}>1$, then choose the maximum in $ i=0,\cdots,M$,
and denote it by $N_l^{n+1} $, and $N_l^{n+1}>1$ holds. Since
$N^{n+1}_l-\frac{C_{\alpha}}{h^2}(N^{n+1}_{l-1}-2N^{n+1}_{l}+N^{n+1}_{l+1})\leq1
$, then
$$\frac{C_{\alpha}}{h^2}(N^{n+1}_{l-1}-2N^{n+1}_{l}+N^{n+1}_{l+1})\geq
N^{n+1}_l-1 >0.$$ So
$$N^{n+1}_{l}<\frac{N^{n+1}_{l-1}+N^{n+1}_{l+1}}{2}.$$ Then  we get $
N^{n+1}_{l}< N^{n+1}_{l-1}$ or $N^{n+1}_{l}< N^{n+1}_{l+1}$, which still hold even at the boundary, including the Dirichlet and Neummann boundaries. This is contradictory with the assumption.

Similarly, we can verify that
$0<P_i^{n+1}\leq L_1$.

\section{Numerical experiments}
We present the simulation results of the schemes (\ref{2.10})-(\ref{2.13}) for Dirichlet boundary and  (\ref{2.10})-(\ref{2.12}) and ($2.13^\prime$) for Neummann boundary to verify all the above theoretical results. In particular, the subdiffusive predator-prey model (\ref{Eq1.1}) with homogeneous Neummann boundary conditions ($1.2^\prime$) is simulated, and the pictures are displayed. Example 5.1 and 5.2 numerically confirm the unconditional stability of the numerical schemes and first order convergence in time for any $\alpha \in (0,1)$. Example 5.3 is for the subdiffusive predator-prey model with specified initial and boundary conditions.

In the computations of Examples 5.1 and 5.2, we take the spacial steplength $h=0.0005$, which is small enough so that the spacial error can be neglected for obtaining convergent rate in time direction. The errors are measured at time $T=1$ and by $l^\infty$ norm. And $\alpha$ is, respectively, taken as $0.3$, $0.6$ and $0.9$.

\begin{example}
For (\ref{2.1})-(\ref{2.3}), we take its  exact analytical solution as
\begin{equation}
u(x,t)=t^2\sin(2\pi x),\label{E1}
\end{equation}
and the non-linear term as
\begin{equation}
f(u)=\frac{1}{u+4}.\label{E2}
\end{equation}
Then, on the right hand side, we need to add the forcing term
\begin{equation}
g(x,t)=\frac{2}{\Gamma(3-\alpha)}t^{(2-\alpha)}\sin(2\pi x)+4\pi^2t^2\sin(2\pi x)-\frac{1}{t^2\sin(2\pi x)+4},\label{E3}
\end{equation}
and the corresponding initial and boundary conditions are respectively
\begin{equation}
u(x,0)=0,\label{E4}
\end{equation}
\begin{equation}
u(0,t)=u(1,t)=0.
\end{equation}
\end{example}


\begin{table}[ht]
\tbl{The error and convergent rate of the proposed scheme for Example 5.1, when $\alpha=0.3$ and $h=0.0005$.}
{
\begin{tabular}{@{}ccc@{}}
 \toprule
 $\tau  $ & $e(h,\tau)$ &rate    \\
 \colrule
$\frac{1}{8}$ & $ 1.628403729893591e-003  $ &   $   $  \\
$\frac{1}{16} $&  $  8.461705972403477e-004  $ &   $ 0.94444$   \\
$\frac{1}{32}$ &  $  4.303180907271331e-004$  &    $0.97555$  \\
$\frac{1}{64}$ &  $ 2.166950795682299e-004$  &   $ 0.98974  $  \\
$\frac{1}{128}$ &  $ 1.088150760031326e-004$  &   $ 0.99379 $ \\
$\frac{1}{256}$ &  $  5.474834928109740e-005$  &   $ 0.99099 $ \\
 \botrule
\end{tabular}
}
\end{table}


\begin{table}[ht]
\tbl{The error and convergent rate of the proposed scheme for Example 5.1, when $\alpha=0.6$ and $h=0.0005$.}
{\begin{tabular}{@{}ccc@{}}
 \toprule
 $\tau  $ & $e(h,\tau)$ & rate  \\
 \colrule
$\frac{1}{8}$ & $  2.021632915774174e-003  $ &   $   $  \\
$\frac{1}{16} $&  $ 9.783209560871864e-004  $ &   $ 1.0471$   \\
$\frac{1}{32}$ &  $ 4.709883249863767e-004$  &    $1.0546$   \\
$\frac{1}{64}$ &  $  2.270387164748922e-004 $  &   $ 1.0528  $ \\
$\frac{1}{128}$ &  $ 1.100565972464995e-004  $  &   $ 1.0447   $ \\
$\frac{1}{256}$ &  $5.382548524024422e-005 $  &   $  1.0319$ \\
 \botrule
\end{tabular}}
\end{table}


\begin{table}[ht]
\tbl{The error and convergent rate of the proposed scheme for Example 5.1, when $\alpha=0.9
$ and $h=0.0005$.} {\begin{tabular}{@{}ccc@{}} \toprule
 $\tau $ & $e(h,\tau)$ &rate  \\
 \colrule
$\frac{1}{8}$ & $ 3.392154303510031e-003  $ &   $   $  \\
$\frac{1}{16} $&  $ 1.657079757972468e-003 $ &   $1.0336 $   \\
$\frac{1}{32}$ &  $ 8.035388508833563e-004$  &    $1.0442$   \\
$\frac{1}{64}$ &  $  3.883816156070585e-004 $  &   $   1.0489$ \\
$\frac{1}{128}$ &  $ 1.876360100631080e-004 $  &   $  1.0495 $ \\
$\frac{1}{256}$ &  $9.083194250991689e-005 $  &   $  1.0467 $ \\ \botrule
\end{tabular}}
\end{table}

\begin{example}
For (\ref{2.1}), (\ref{2.2}) and ($2.3^\prime$), the exact solution and the boundary condition are, respectively, taken as $t^2\cos(2\pi x)$ and
\[\frac{\partial u(x,t)}{\partial x}|_{x=0}=\frac{\partial u(x,t)}{\partial x}|_{x=1}=0.\]
 We still use (\ref{E2}) and (\ref{E4}) as the non-linear term and initial condition, respectively.
 And the following forcing term is needed to add to the right hand side of the equation,
\begin{equation}
g(x,t)=\frac{2}{\Gamma(3-\alpha)}t^{(2-\alpha)}\cos(2\pi x)+4\pi^2t^2\cos(2\pi x)-\frac{1}{t^2\cos(2\pi x)+4}.
\end{equation}

\begin{table}[ht]
\tbl{The error and convergent rate of the proposed scheme for Example 5.2, when $\alpha=0.3$ and $h=0.0005$.}
 {
 \begin{tabular}{@{}ccc@{}}
  \toprule
 $\tau  $ & $e(h,\tau)$ &rate  \\
 \colrule
$\frac{1}{8}$ & $ 2.887575706533641e-003   $ &   $   $  \\
$\frac{1}{16} $&  $    1.527198871648983e-003$ &   $0.91897   $   \\
$\frac{1}{32}$ &  $ 7.855525404816266e-004  $  &    $0.95911  $   \\
$\frac{1}{64}$ &  $  3.983144175851994e-004  $  &   $ 0.97980   $ \\
$\frac{1}{128}$ &  $  2.006793113920047e-004 $  &   $0.98902   $ \\
$\frac{1}{256}$ &  $ 1.009537078571210e-004 $  &   $   0.99120 $ \\ \botrule
\end{tabular}}
\end{table}

\begin{table}[ht]
\tbl{The error and convergent rate of the proposed scheme for Example 5.2, when $\alpha=0.6$ and $h=0.0005$.}
 {
 \begin{tabular}{@{}cccc@{}}
  \toprule
 $\tau $ & $e(h,\tau)$ &rate  \\
 \colrule
$\frac{1}{8}$ & $  2.713892842113319e-003  $ &   $   $  \\
$\frac{1}{16} $&  $  1.337146145350632e-003 $ &   $ 1.0212  $   \\
$\frac{1}{32}$ &  $ 6.550178686808295e-004$  &    $1.0296 $   \\
$\frac{1}{64}$ &  $  3.205477251759792e-004   $  &   $1.0310     $ \\
$\frac{1}{128}$ &  $  1.572572233041747e-004 $  &   $  1.0274   $ \\
$\frac{1}{256}$ &  $ 7.755217692539951e-005 $  &   $   1.0199  $ \\ \botrule
\end{tabular}}
\end{table}

\begin{table}[ht]
\tbl{The error and convergent rate of the proposed scheme for Example 5.2, when $\alpha=0.9$ and $h=0.0005$.}
 {
 \begin{tabular}{@{}ccc@{}}
  \toprule
 $\tau  $ & $e(h,\tau)$ &rate  \\
 \colrule
$\frac{1}{8}$ & $   3.769255105664726e-003 $ &   $   $  \\
$\frac{1}{16} $&  $   1.832929458344568e-003 $ &   $1.0401   $   \\
$\frac{1}{32}$ &  $ 8.885772877893494e-004 $  &    $1.0446 $   \\
$\frac{1}{64}$ &  $   4.302894303225280e-004 $  &   $ 1.0462    $ \\
$\frac{1}{128}$ &  $  2.084713949954686e-004 $  &   $ 1.0455   $ \\
$\frac{1}{256}$ &  $ 1.012302249301378e-004 $  &   $  1.0422    $ \\ \botrule
\end{tabular}}
\end{table}
\end{example}

\begin{example}
Consider the reaction diffusion equation\cite{aly2011}
\begin{eqnarray}
D^{\alpha}_{t} N &= & d_1 \frac{\partial^2 N}{\partial x^2}+N(1-N-\frac{a P}{P+N}),\nonumber\\
D^{\alpha}_t P &= & d_2\frac{\partial^2 P}{\partial x^2}+\sigma P(-\frac{\gamma+\delta \beta P}{1+\beta P}+\frac{N}{P+N}),\label{e2}
\end{eqnarray}
with the homogeneous Neummann boundary conditions on the domain $\Omega=[0,1]$.
\end{example}

Let us denote  $f(N,P)=N(1-N-\frac{a P}{P+N})$, $g(N,P)=\sigma P(-\frac{\gamma+\delta \beta P}{1+\beta P}+\frac{N}{P+N})$, and define $(\bar{N},\bar{P})$ as the equilibrium point of (\ref{e2}).  In the case, $\sigma=1,a=1.1, \gamma=0.05, \beta=1$, and $\delta=0.5$, then as
$f(\bar{N},\bar{P})=0$ and $g(\bar{N},\bar{P})=0$, we can obtain the equilibrium point $(\bar{N},\bar{P})=(0.113585,0.471397)$.
The simulations were performed for the system on a fixed grid
with spatial stepsize $h=0.005$ and time stepsize $\tau=0.1$. As the initial
condition, we use
\begin{align}
&N(x,0)=\bar{N}+0.0214\cos(\pi x),\nonumber\\
&P(x,0)=\bar{P}+0.0066\cos(\pi x).\nonumber
\end{align}

We focus predominantly on displaying the properties of the numerical solutions for different time fractional order $\alpha$, see Fig. 1, Fig. 2, and Fig. 3.

\begin{figure}[pb]
\centerline{\psfig{file=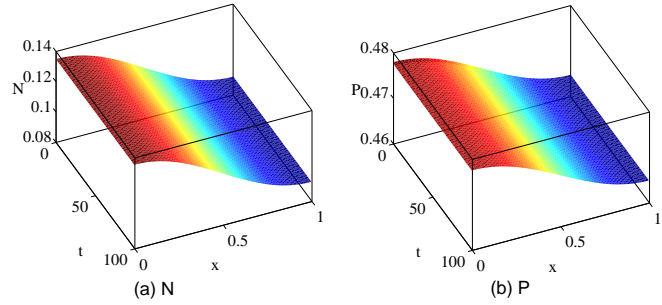,width=4.0in}}
\vspace*{8pt}
\caption{Numerical solution for $d_1=0.005, d_2=0.2, \alpha=0.2$.}
\end{figure}

\begin{figure}[pb]
\centerline{\psfig{file=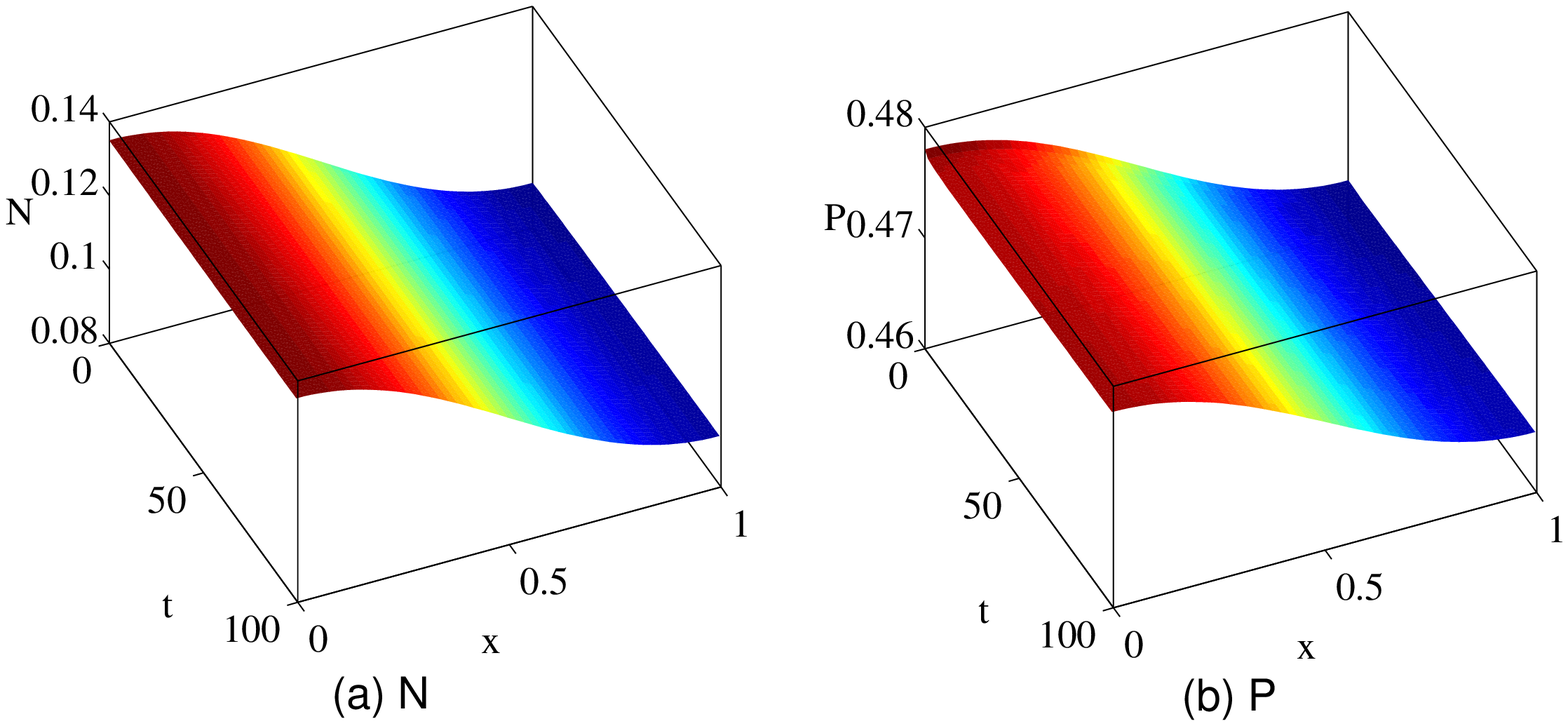,width=4.0in}}
\vspace*{8pt}
\caption{Numerical solution for $d_1=0.005, d_2=0.2, \alpha=0.5$.}
\end{figure}

\begin{figure}[pb]
\centerline{\psfig{file=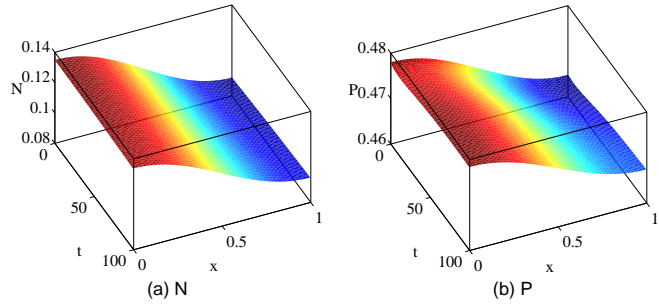,width=4.0in}}
\vspace*{8pt}
\caption{Numerical solution for $d_1=0.005, d_2=0.2, \alpha=0.9$.}
\end{figure}

\section{Conclusions}
We introduce the unconditional stable semi-implicit numerical schemes for subdiffusive reaction diffusion equation with Dirichlet boundary condition and Neummann boundary condition, respectively. And the subdiffusive predator-prey model is detailedly discussed. We prove that its analytical solution is positive and bounded. Then we show that the proposed numerical schemes preserve the positivity and boundedness of the analytical solutions. The extensive numerical experiments are performed to confirm the theoretical results and show the dissipative properties of subdiffusive predator-prey model.

\section*{Acknowledgements}
This work was supported by the Program for New Century Excellent Talents in University under Grant No. NCET-09-0438, the
National Natural Science Foundation of China under Grant No. 10801067 and No. 11271173, and the Fundamental Research Funds for the Central
Universities under Grant No. lzujbky-2010-63 and No. lzujbky-2012-k26.

\appendix

\section{}

 Proof of Theorems $4.2$:  Taking the initial iteration function as
 \begin{align}
 (\bar{u}^{(0)}_1,\bar{u}^{(0)}_2)=(\tilde{u}_1,\tilde{u}_2),\nonumber\\
 (\underbar {u}^{(0)}_1,\underbar{u}^{(0)}_2)=(\utilde{u}_1,\utilde{u}_2),\nonumber
 \end{align}
 with $\utilde{u}_1 \leqslant  \tilde{u}_1$ and $\utilde{u}_2 \leqslant  \tilde{u}_2$, define the following iteration
  \begin{equation}\label{A1}
    \left\{
     \begin{aligned}
     &\frac{\partial^{\alpha}\bar{u}_1^{(k)}}{\partial t^{\alpha}}-\frac{\partial^2 \bar{u}_1^{(k)}}{\partial x^2}+L\cdot \bar{u}^{(k)}_1=L\cdot \bar{u}^{(k-1)}_1+f_1(\bar{u}_1^{(k-1)},\underbar{u}_2^{(k-1)})  \\
   &\frac{\partial^{\alpha}\bar{u}_2^{(k)}}{\partial t^{\alpha}}-\frac{\partial^2 \bar{u}_2^{(k)}}{\partial x^2}+L\cdot
  \bar{u}^{(k)}_2=L\cdot \bar{u}^{(k-1)}_2+f_2(\bar{u}_1^{(k-1)},\bar{u}_2^{(k-1)})  \\
   &\frac{\partial^{\alpha}\underbar{u}_1^{(k)}}{\partial t^{\alpha}}-\frac{\partial^2 \underbar{u}_1^{(k)}}{\partial x^2}+L\cdot
   \underbar{u}^{(k)}_1=L\cdot \underbar{u}^{(k-1)}_1+f_1(\underbar{u}_1^{(k-1)},\bar{u}_2^{(k-1)})  \\
   &\frac{\partial^{\alpha}\underbar{u}_2^{(k)}}{\partial t^{\alpha}}-\frac{\partial^2 \underbar{u}_2^{(k)}}{\partial x^2}+L\cdot
   \underbar{u}^{(k)}_2=L\cdot \underbar{u}^{(k-1)}_2+f_2(\underbar{u}_1^{(k-1)},\underbar{u}_2^{(k-1)})  \\
   & B\bar{u}_i^{(k)} |_{\partial\Omega\times(0,T]}=B \underbar{u}_i^{(k)} |_{\partial\Omega\times(0,T]}
   =g_i(x,t)|_{\partial\Omega\times(0,T]},\,\,(i=1,2) \\
   &\bar{u}_i^{(k)}(x,0)=\underbar{u}_i^{(k)}(x,0)=\varphi_i(x),\,\,(x\in\bar{\Omega},\,i=1,2),
     \end{aligned}
     \right.
    \end{equation}
    where $L$ is the maximum of Lipschiz constants of $f_1$ and $f_2$.
Subtracting  (\ref{4.4.3}) and (\ref{4.4.4}) from  (\ref{A1}) leads to
     \begin{equation}
    \left\{
     \begin{aligned}
     &\frac{\partial^{\alpha}(\bar{u}_1^{(1)}-\bar{u}_1^{(0)})}{\partial t^{\alpha}}-\frac{\partial^2 (\bar{u}_1^{(1)}-\bar{u}_1^{(0)})}{\partial x^2}+L\cdot (\bar{u}^{(1)}_1-\bar{u}^{(0)}_1)\leq0  \nonumber\\
   &\frac{\partial^{\alpha}(\bar{u}_2^{(1)}-\bar{u}_2^{(0)})}{\partial t^{\alpha}}-\frac{\partial^2 (\bar{u}_2^{(1)}-\bar{u}_2^{(0)})}{\partial x^2}+L\cdot
  (\bar{u}_2^{(1)}-\bar{u}_2^{(0)})\leq0 \nonumber\\
   &\frac{\partial^{\alpha}(\underbar{u}_1^{(0)}-\underbar{u}_1^{(1)})}{\partial t^{\alpha}}-\frac{\partial^2 (\underbar{u}_1^{(0)}-\underbar{u}_1^{(1)})}{\partial x^2}+L\cdot
   (\underbar{u}_1^{(0)}-\underbar{u}_1^{(1)})\leq0 \nonumber\\
   &\frac{\partial^{\alpha}(\underbar{u}_2^{(0)}-\underbar{u}_2^{(1)})}{\partial t^{\alpha}}-\frac{\partial^2 (\underbar{u}_2^{(0)}-\underbar{u}_2^{(1)})}{\partial x^2}+L\cdot
   (\underbar{u}_2^{(0)}-\underbar{u}_2^{(1)})\leq0 \nonumber\\
   & B(\bar{u}_i^{(1)}-\bar{u}_i^{(0)}) |_{\partial\Omega\times(0,T]}=B (\underbar{u}_i^{(0)}-\underbar{u}_i^{(1)}) |_{\partial\Omega\times(0,T]}
   =0,\,\,(i=1,2) \nonumber\\
   &\bar{u}_i^{(1)}(x,0)-\bar{u}_i^{(0)}(x,0)=\underbar{u}_i^{(0)}(x,0)-\underbar{u}_i^{(1)}(x,0)=0,\,\,(x\in\bar{\Omega},\,i=1,2).\nonumber
     \end{aligned}
     \right.
    \end{equation}
    According to the maximum principle Theorem \ref{Th4.1} and Remark \ref{Th4.1}, we know that
    \[\bar{u}^{(1)}_i\leq\bar{u}^{(0)}_i,\underbar{u}^{(0)}_i\leq\underbar{u}^{(1)}_i\,\,(k=0,1,\cdots). \]
   Supposing $\bar{u}^{(k)}_i\leq\bar{u}^{(k-1)}_i,\,\underbar{u}^{(k-1)}_i\leq\underbar{u}^k_i $, and then that the nonlinear term $f_1$ is quasi-monotone decreasing results in
   \begin{equation}\label{Af1}
   \begin{array}{ll}
\displaystyle
   f_1(\bar{u}_1^{(k)},\underbar{u}_2^{(k)})-f_1(\bar{u}_1^{(k-1)},\underbar{u}_2^{(k-1)})
   \\
   \\
   \displaystyle=f_1(\bar{u}_1^{(k)},\underbar{u}_2^{(k)})-f_1(\bar{u}_1^{(k)},\underbar{u}_2^{(k-1)})
   +f_1(\bar{u}_1^{(k)},\underbar{u}_2^{(k-1)}) -f_1(\bar{u}_1^{(k-1)},\underbar{u}_2^{(k-1)})
   \\
   \\
   \displaystyle\leq0+L\cdot|\bar{u}_1^{(k)} - \bar{u}_1^{(k-1)}|\\
   \\
   \displaystyle\leq L\cdot(\bar{u}_1^{(k-1)} - \bar{u}_1^{(k)}).
   \end{array}
   \end{equation}
   In a similar way we also get
   \begin{equation}\label{Af2}
   f_2(\bar{u}_1^{(k)},\bar{u}_2^{(k)})-f_2(\bar{u}_1^{(k-1)},\bar{u}_2^{(k-1)})\leq L\cdot(\bar{u}_2^{(k-1)} - \bar{u}_2^{(k)}),
   \end{equation}
   \begin{equation}\label{Af3}
   f_1(\underbar{u}_1^{(k )},\bar{u}_2^{(k )})-f_1(\underbar{u}_1^{(k-1)},\bar{u}_2^{(k-1)})\leq L\cdot(\underbar{u}_1^{(k)}-\underbar{u}_1^{(k-1)}),
   \end{equation}
   \begin{equation}\label{Af4}
   f_2(\underbar{u}_1^{(k )},\underbar{u}_2^{(k )})-f_2(\underbar{u}_1^{(k-1)},\underbar{u}_2^{(k-1)})\leq L\cdot(\underbar{u}_2^{(k)}-\underbar{u}_2^{(k-1)}).
   \end{equation}

  So, together with (\ref{Af1})-(\ref{Af4}), the iteration (\ref{A1}) implies
  \begin{equation*}
    \left\{
    \begin{array}{lll}
     \displaystyle\frac{\partial^{\alpha}(\bar{u}_1^{(k+1)}-\bar{u}_1^{(k)})}{\partial t^{\alpha}}-\frac{\partial^2 (\bar{u}_1^{(k+1)}-\bar{u}_1^{(k)})}{\partial x^2}+L\cdot (\bar{u}^{(k+1)}_1-\bar{u}^{(k)}_1)\\
     \\
     \displaystyle~~~~~~~~~~~~~~\leq L\cdot (\bar{u}^{(k)}_1-\bar{u}^{(k-1)}_1)+f_1(\bar{u}_1^{(k)},\underbar{u}_2^{(k)})-f_1(\bar{u}_1^{(k-1)},\underbar{u}_2^{(k-1)})\leq0 \\
     \\
   \displaystyle\frac{\partial^{\alpha}(\bar{u}_2^{(k+1)}-\bar{u}_2^{(k)})}{\partial t^{\alpha}}-\frac{\partial^2 (\bar{u}_2^{(k+1)}-\bar{u}_2^{(k)})}{\partial x^2}+L\cdot
  (\bar{u}_2^{(k+1)}-\bar{u}_2^{(k)})\\
     \\
     \displaystyle~~~~~~~~~~~~~~\leq L\cdot(\bar{u}_2^{(k)} - \bar{u}_2^{(k-1)})+ f_2(\bar{u}_1^{(k)},\bar{u}_2^{(k)})-f_2(\bar{u}_1^{(k-1)},\bar{u}_2^{(k-1)})\leq0 \\
     \\
   \displaystyle\frac{\partial^{\alpha}(\underbar{u}_1^{(k)}-\underbar{u}_1^{(k+1)})}{\partial t^{\alpha}}-\frac{\partial^2 (\underbar{u}_1^{(k)}-\underbar{u}_1^{(k+1)})}{\partial x^2}+L\cdot
   (\underbar{u}_1^{(k)}-\underbar{u}_1^{(k+1)})\\
   \\
   \displaystyle~~~~~~~~~~~~~~\leq L\cdot(\underbar{u}_1^{(k-1)}-\underbar{u}_1^{(k)})+ f_1(\underbar{u}_1^{(k )},\bar{u}_2^{(k )})-f_1(\underbar{u}_1^{(k-1)},\bar{u}_2^{(k-1)})\leq0 \\\\
   \displaystyle\frac{\partial^{\alpha}(\underbar{u}_2^{(k)}-\underbar{u}_2^{(k+1)})}{\partial t^{\alpha}}-\frac{\partial^2 (\underbar{u}_2^{(k)}-\underbar{u}_2^{(k+1)})}{\partial x^2}+L\cdot
   (\underbar{u}_2^{(k)}-\underbar{u}_2^{(k+1)})\\
   \\
   \displaystyle~~~~~~~~~~~~~~\leq L\cdot(\underbar{u}_2^{(k-1)}-\underbar{u}_2^{(k)})+f_2(\underbar{u}_1^{(k )},\underbar{u}_2^{(k )})-f_2(\underbar{u}_1^{(k-1)},\underbar{u}_2^{(k-1)})\leq0 \\\\
   \displaystyle B(\bar{u}_i^{(k+1)}-\bar{u}_i^{(k)}) |_{\partial\Omega\times(0,T]}=B (\underbar{u}_i^{(k)}-\underbar{u}_i^{(k+1)}) |_{\partial\Omega\times(0,T]}
   =0,\,\,(i=1,2) \\\\
   \displaystyle\bar{u}_i^{(k+1)}(x,0)-\bar{u}_i^{(k)}(x,0)=\underbar{u}_i^{(k)}(x,0)-\underbar{u}_i^{(k+1)}(x,0)=0,\,\,(x\in\bar{\Omega},\,i=1,2).
    \end{array}
     \right.
    \end{equation*}
Then there exists
  \[\bar{u}^{(k+1)}_i\leq\bar{u}^{(k)}_i,\,\underbar{u}^{(k)}_i\leq\underbar{u}^{(k+1)}_i\,\,(k=0,1,\cdots). \]
Recalling the iteration (\ref{A1}) again, we deduce
  \begin{equation}
    \left\{
     \begin{array}{ll}
     \displaystyle\frac{\partial^{\alpha}(\underbar{u}^{(k)}_1-\bar{u}^{(k)}_1)}{\partial t^{\alpha}}-\frac{\partial^2 (\underbar{u}^{(k)}_1-\bar{u}^{(k)}_1)}{\partial x^2}+L\cdot (\underbar{u}^{(k)}_1-\bar{u}^{(k)}_1)\\
     \\
     \displaystyle~~~~\leq L\cdot (\underbar{u}^{(k-1)}_1-\bar{u}^{(k-1)}_1)+f_1(\underbar{u}_1^{(k-1)},\bar{u}_2^{(k-1)})-f_1(\bar{u}_1^{(k-1)},\underbar{u}_2^{(k-1)})\leq0  \\
     \\
   \displaystyle\frac{\partial^{\alpha}(\underbar{u}^{(k)}_2-\bar{u}^{(k)}_2)}{\partial t^{\alpha}}-\frac{\partial^2 (\underbar{u}^{(k)}_2-\bar{u}^{(k)}_2)}{\partial x^2}+L\cdot(\underbar{u}^{(k)}_2-\bar{u}^{(k)}_2)\\
   \\
   \displaystyle~~~~\leq L\cdot (\underbar{u}^{(k-1)}_2-\bar{u}^{(k-1)}_2)+f_2(\underbar{u}_1^{(k-1)},\underbar{u}_2^{(k-1)})-f_2(\bar{u}_1^{(k-1)},\bar{u}_2^{(k-1)})\leq0  \\
   \\
   \displaystyle B (\underbar{u}^{(k)}_i-\bar{u}^{(k)}_i) |_{\partial\Omega\times(0,T]}=0, \,\,(i=1,2)  \\
   \\
   \displaystyle \underbar{u}^{(k)}_i(x,0)-\bar{u}^{(k)}_i(x,0) =0,\,\,(x\in\bar{\Omega},\,i=1,2),
     \end{array}
     \right.
    \end{equation}
then 
$$\underbar{u}^{(k)}_i\leq\bar{u}^{(k)}_i.$$
So
$$
 \utilde{u}_i\leq \underbar{u}^{(1)}_i\leq\cdots\leq\underbar{u}^{(k)}_i\leq\bar{u}^{(k)}_i\leq\cdots
  \leq\bar{u}^{(1)}_i\leq\tilde{u}_i, \,\,   (i=1,2).
$$
Note that $f_1$ is quasi-monotone decreasing and $f_2$ is quasi-monotone increasing, then there are
 \begin{align}
 &\lim_{k\rightarrow+\infty}\bar{u}^{(k)}_i=\bar{u}_i(x,t),\nonumber\\
 &\lim_{k\rightarrow+\infty}\underbar{u}^{(k)}_i=\underbar{u}_i(x,t),\,\,(i=1,2),\nonumber
 \end{align}
 which satisfy $\bar{u}_i\geqslant\underbar{u}_i$ and
 \begin{equation}
   \left\{
    \begin{aligned}
 &\frac{\partial^{\alpha}\bar{u}_1}{\partial t^{\alpha}}-\frac{\partial^2 \bar{u}_1}{\partial x^2}-f_1(\bar{u}_1,\underbar{u}_2)=0 \nonumber\\
 &\frac{\partial^{\alpha}\bar{u}_2}{\partial t^{\alpha}}-\frac{\partial^2 \bar{u}_2}{\partial x^2}-f_2(\bar{u}_1,\bar{u}_2)= 0 \nonumber\\
 &\frac{\partial^{\alpha}\underbar{u}_1}{\partial t^{\alpha}}-\frac{\partial^2\underbar{u}_1 }{\partial x^2}-f_1(\underbar{u}_1,\bar{u}_2)=0 \nonumber\\
 &\frac{\partial^{\alpha}\underbar{u}_2}{\partial t^{\alpha}}-\frac{\partial^2\underbar{u}_2 }{\partial x^2}-f_2(\underbar{u}_1,\underbar{u}_2)=0 \nonumber\\
 &B\bar{u}_i|_{\partial\Omega\times(0,T]}=B \underbar{u}_i |_{\partial\Omega\times(0,T]}=g_i(x,t) \nonumber\\
 &\bar{u}_i(x,0)=\underbar{u}_i(x,0)=\varphi_i(x),\,\,(x\in\bar{\Omega},\,i=1,2).\nonumber
       \end{aligned}
      \right.
     \end{equation}
Next we will certify
 \[\bar{u}_i=\underbar{u}_i=u_i,(i=1,2).\]
Defining  $w_1=\bar{u}_1-\underbar{u}_1$,  $w_2=\bar{u}_2-\underbar{u}_2$,  we have known $w_1\geq0$ and $w_2\geq0$ from above discussions.
 According to the iteration, we can obtain

\begin{align}
  \frac{\partial^{\alpha}w_1}{\partial t^{\alpha}}-\frac{\partial^2 w_1}{\partial x^2} &=
  f_1(\bar{u}_1,\underbar{u}_2)-f_1(\underbar{u}_1,\bar{u}_2)\nonumber \\
  &= f_1(\bar{u}_1,\underbar{u}_2)-f_1(\underbar{u}_1,\underbar{u}_2)
  +f_1(\underbar{u}_1,\underbar{u}_2)-f_1(\underbar{u}_1,\bar{u}_2)\nonumber \\
  &\leq  L\cdot(\bar{u}_1-\underbar{u}_1)+ L\cdot(\bar{u}_2-\underbar{u}_2) \nonumber\\
  &= L\cdot(w_1+w_2),\nonumber
  \end{align}

 \begin{align}
  \frac{\partial^{\alpha}w_2}{\partial t^{\alpha}}-\frac{\partial^2 w_2}{\partial x^2}& =
  f_2(\bar{u}_1,\bar{u}_2)-f_2(\underbar{u}_1,\underbar{u}_2)\nonumber\\
  &= f_2(\bar{u}_1,\bar{u}_2)-f_2(\bar{u}_1,\underbar{u}_2)
  +f_2(\bar{u}_1,\underbar{u}_2)-f_2(\underbar{u}_1,\underbar{u}_2)\nonumber\\
  &\leq  L\cdot(\bar{u}_2-\underbar{u}_2)+ L\cdot(\bar{u}_1-\underbar{u}_1) \nonumber\\
  &= L\cdot(w_1+w_2).\nonumber
  \end{align}

 So
 \begin{equation*}
   \left\{
    \begin{aligned}
      &\frac{\partial^{\alpha}w_1}{\partial t^{\alpha}}  - \frac{\partial^2 w_1}{\partial x^2} -L\cdot(w_1+w_2)\leq0\\
       &\frac{\partial^{\alpha}w_2}{\partial t^{\alpha}}  - \frac{\partial^2 w_2}{\partial x^2} -L\cdot(w_1+w_2)\leq0\\
      &B w_1 =0, B w_2 =0     \\
      & w_1(x,0)=0,  w_2(x,0)=0.
       \end{aligned}
      \right.
     \end{equation*}
Denoting $w=w_1+w_2 \ge 0$, it's obvious that
   \begin{equation} \label{4.7}
   \left\{
    \begin{aligned}
      &\frac{\partial^{\alpha}w}{\partial t^{\alpha}}  - \frac{\partial^2 w}{\partial x^2} -L\cdot w\leq0\\
      &B w =0      \\
      & w(x,0)=0.
       \end{aligned}
      \right.
     \end{equation}
Based on (\ref{4.7}),  next we try to prove that $w \equiv0$ in $\overline{\Omega}_T$. Suppose that $w$ obtains its maximum value at $(\hat{x},\hat{t})$, if $(\hat{x},\hat{t}) \in \Gamma_T$  and the boundary conditions are Dirichlet's, $w \equiv0$ holds obviously; if $(\hat{x},\hat{t}) \in \Gamma_T$ and $u(\hat{x},\hat{t})$ is strictly bigger than $u(x,t)$ for any $(x,t) \in \Omega_T$, and the boundary conditions are Neummann's, it can be shown that $w \equiv0$ holds by the ideas in Remark \ref{Remark4.1} and the following proof.

Now assume $(\hat{x},\hat{t}) \in \Omega_T$ and $w(\hat{x},\hat{t})>0$, then there exists $t^*\, (<\hat{t}) $ such that
 $w(x,t)<\frac{1}{2}w(\hat{x},\hat{t})$ for any $t\in(0,t^{*})$ and $x \in \overline{\Omega}$.
Introduce the function $u$ of $t$ such that it satisfies
  $$
  \left\{
    \begin{array}{l}
     \displaystyle\frac{ \partial u(t)}{\partial t}=-\frac{1}{2}\frac{w(\hat{x},\hat{t})}{t^{*}}\quad {\rm in}\,\,(0,t^{*}),\\
     \\
     \displaystyle u(0)=\frac{1}{2}w(\hat{x},\hat{t}),
   \end{array}
      \right.
  $$
and $u(t) =0, \, t \in [t^*,T]$, denote
 \begin{equation*}
    \bar{w}(x,t)=w(x,t)+u(t)\quad {\rm in} \,~ \Omega\times[0,T].
     \end{equation*}
 Then the maximum of $\bar{w}$ is still at $(\hat{x},\hat{t})$, and $\bar{w}$ satisfies the following inequality
  \begin{equation*}
    \frac{\partial^{\alpha}\bar{w}}{\partial t^{\alpha}}- \frac{\partial^2 \bar{w}}{\partial x^2}\leq L w+\frac{ \partial^{\alpha} u(t)}{\partial t^{\alpha}}.
     \end{equation*}
At $(\hat{x}, \hat{t})$, it follows that $\frac{\partial^{\alpha} \bar{w} (\hat{x},t)}{\partial t^{\alpha}}|_{t=\hat{t}}\geq0$,
 and $- \frac{\partial^2 \bar{w}(x,\hat{t})}{\partial x^2}|_{x=\hat{x}}\geq0$.
While
 \begin{equation*}
 \begin{array}{lll}
\displaystyle
 \frac{ \partial^{\alpha} u(t)}{\partial t^{\alpha}}|_{t=\hat{t}} &=& \displaystyle \frac{1}{\Gamma(1-\alpha)}\int_0^{\hat{t}}(\hat{t}-\tau)^{-\alpha}\frac{du(\tau)}{d\tau}d\tau\\
\\
 &=& \displaystyle-\frac{1}{2\Gamma(1-\alpha)}\int_0^{t^{*}}(\hat{t}-\tau)^{-\alpha}\frac{w(\hat{x},\hat{t})}{t^{*}}d\tau
\\
\\
 &=&\displaystyle -\frac{w(\hat{x},\hat{t})}{2\Gamma(1-\alpha)t^{*}}\int_0^{t^{*}}(\hat{t}-\tau)^{-\alpha}d\tau
\\
\\
 &<& \displaystyle -\frac{w(\hat{x},\hat{t})}{2\Gamma(1-\alpha)t^{*}}\int_0^{t^{*}}\hat{t}^{-\alpha}d\tau
\\
\\
 &=& \displaystyle -\frac{w(\hat{x},\hat{t})}{2\Gamma(1-\alpha)}\hat{t}^{-\alpha} .
\end{array}
\end{equation*}
Since $\hat{t} \leq T$, if $T\leq(\frac{1}{2L\Gamma(1-\alpha)})^\frac{1}{\alpha}$, then $T^\alpha\leq\frac{1}{2L\Gamma(1-\alpha)}$ and
\begin{equation*}
    L w(\hat{x},\hat{t})+\frac{ \partial^{\alpha} u(\hat{t})}{\partial t^{\alpha}}<0.
     \end{equation*}
We arrive at a contradiction, so that $w\equiv 0$ holds.
%

If $(\frac{1}{2L\Gamma(1-\alpha)})^\frac{1}{\alpha} < T \leq2(\frac{1}{2L\Gamma(1-\alpha)})^\frac{1}{\alpha}$, we take $\tilde{T} =(\frac{1}{2L\Gamma(1-\alpha)})^\frac{1}{\alpha}$.
In the domain $\Omega\times[0,\tilde{T}]$, we can get $w \equiv 0$.
For $\tilde{T}\leq t\leq T$, set $\tilde{t}=t-\tilde{T}$ and $\tilde{w}(\tilde{t})=w(\tilde{t}+\tilde{T})$.
So when   $ 0\leq\tilde{t}\leq T-\tilde{T}$,   $\tilde{w}$ satisfy
\begin{equation*}
   \left\{
    \begin{aligned}
      &\frac{\partial^{\alpha}\tilde{w}}{\partial t^{\alpha}}-\frac{\partial^2 \tilde{w}}{\partial x^2} -L\cdot \tilde{w}\leq0 \quad \\
      &B \tilde{w} =0      \\
      & \tilde{w}(x,0)=0.
       \end{aligned}
      \right.
     \end{equation*}
Since $ T-\tilde{T} \leq(\frac{1}{2L\Gamma(1-\alpha)})^\frac{1}{\alpha}$, we have $\tilde{w}=0$ in $\Omega\times[0, T-\tilde{T}]$.
 Then $w=0$ in $\Omega\times[\tilde{T},T]$.
Consequently $w=0$ in $\Omega\times[0, T]$. This process can be continued for any finite times, so we obtain $w\equiv 0$ in $\overline{\Omega}_T$ for any $T$.

Combining with the known $w_1\geq0$ and  $w_2\geq0$, $w\equiv0$ implies $w_1\equiv0$  and $w_2\equiv0$.
Then we get  $\bar{u}_1=\underbar{u}_1$  and   $\bar{u}_2=\underbar{u}_2$.  Setting
$$
u_i(x,t)=\bar{u}_i=\underbar{u}_i,\,\, i=(1,2),
$$
 then $u(x,t)=(u_1,u_2)$ solves (\ref{4.3}).

Next we prove the uniqueness of the solution.

Suppose that there exists another solution $u'(x,t)=(u'_1,u'_2)$ which satisfies $V(x,t)\leq u'(x,t)\leq U(x,t)$.
Obviously, $\underbar{u}_i^{(0)}\leq  u'_i\leq \bar{u}_i^{(0)}\,(i=1,2)$. On the one hand,
since $u'(x,t)$ can be considered as an upper solution, according to  $V(x,t)\leq u'(x,t)$,
we know $\underbar{u}_i^{(k)}\leq  u'_i$. Therefore $\lim_{k\rightarrow \infty} \underbar{u}_i^{(k)}(x,t)=u_i\leq u'_i$.
On the other hand, $u'(x,t)$ can also be considered as a lower solution, according to $u'(x,t)\leq U(x,t)$, we know $\lim_{k\rightarrow \infty} \bar{u}_i^{(k)}(x,t)=u_i\geq u'_i$.
Consequently $u'_i=u_i$.



\end{document}